\documentclass[twoside]{amsart}
\usepackage{amssymb,verbatim}
\usepackage{amsmath}
\usepackage{amsthm}
\usepackage{graphicx}
\usepackage{color}
\usepackage[all]{xy}

\newtheorem{theorem}{Theorem}[section]

\newtheorem{Lemma}[theorem]{Lemma}

\newtheorem{Definition}[theorem]{Definition}

\newcommand{\F}{\mathbb F}

\DeclareMathOperator{\Proj}{Proj}

\begin{document}


\title[On $k$-PN Functions over Finite Fields]{On $k$-PN Functions over Finite Fields}
\author{Marco Pizzato}

\thanks{The author would like to thank his PhD advisor, Sandro Mattarei.}

\address{Dipartimento di Matematica, Universit\`a di Trento, I-38123
Povo (TN)} 
\email{marco.pizzato1@gmail.com}


\begin{abstract} 
Starting from PN functions, we introduce the concept of $k$-PN functions and classify $k$-PN monomials over finite fields of order $p,p^2$ and $p^4$ for \textit{small} values of $k$.
\end{abstract}

\maketitle

\section{Introduction}
In this article we present some results concerning $k$-PN functions, which we define as following.

\begin{Definition}
Let $f$ be a polynomial over $\F_q$. We define the $k$-th finite difference in \textit{directions} $(a_1,\dots,a_k)$ as: 
$$\nabla_{a_1,\dots,a_k}^k f = g(x+a_k)-g(x),$$ where $g(x) = \nabla_{a_1,\dots,a_{k-1}}^{k-1} f$ and $\nabla_{a}^1 f = f(x+a)-f(x)$. Now we say that a polynomial $f \in \F_q[x]$ is $k$-PN over $\F_q$ if the function associated to $\nabla_{a_1,\dots,a_k}^k f$ is a bijection for every choice of $(a_1,\dots,a_k)$, with $a_i \in \F_q, a_i \neq 0$ for all $i$.
\end{Definition}

We note that in the case $k = 1$ we obtain PN functions, well known and studied. We will focus our attention on $k$-PN monomials over \textit{small} finite fields, generalizing some results obtained for PN monomials. In \cite{MR1030384}, \cite{MR914241}, \cite{MR1008158} and \cite{MR1045927} we have the first important results, showing that essentially only quadratic polynomials are PN functions over prime fields of odd characteristic. In \cite{MR2231922}, the author shows that, up to $p$-th powers, the only PN monomial over the field $\F_{p^2}$ is $x^2$. In  \cite{MR2890555}, the authors shows that the previous result holds also for $\F_{p^4}$, provided that $p \geq 5$.
\medskip

We present now our main results, generalizing the previous ones.

\begin{theorem}\label{teo1}
Suppose $x^n$ is $k$-PN over $\F_{p}$, where $n \leq p-1$ and $p \geq k+2$. Then $n = k+1$.
\end{theorem}

\begin{theorem}\label{teo2}
Suppose $x^n$ is a $k$-PN monomial over $\F_{p^2}$ with $p \geq 2k + 2$ and $n \leq p^2-1$. Then, writing $n = a + bp$, we have $a + b = k+1$.
\end{theorem}

We can obtain a more complete result for $k = 2$ and $k=3$.

\begin{theorem}\label{teo3}
Suppose $x^n$ is $2$-PN over $\F_{p^2}, p \geq 5$, and $n \leq p^2-1$. Then $n\in\{3,3p\}$ if $p\equiv -1\pmod{3}$, and $n\in\{3,3p,1+2p,2+p\}$  if $p\equiv 1\pmod{3}$. If $x^n$ is $3$-PN over $\F_{p^2}$ and $n \leq p^2-1$, then $n \in \{4, 4p \}$. 
\end{theorem}

We conclude with results for $\F_{p^4}$.

\begin{theorem}\label{teo4}
Let $f = x^n$ be a $2$-PN monomial over $\F_{p^4}, p \ge 5$, $\deg(f) \le p^4-1$, and suppose $n$ not divisible by $p$. Then $n \in \{3,2+p^2, 1+2p^2\}$.
\end{theorem}

\begin{theorem}\label{teo5}
Let $f = x^n$ be a $3$-PN monomial over $\F_{p^4}, p \ge 5$, $\deg(f) \le p^4-1$, and suppose $n$ not divisible by $p$. Then $n = 4$.
\end{theorem}

In the second section we present the tools that we need in order to obtain proofs of the previous theorems, effort that we accomplish in the last section.

\section{Preliminary Results}
\label{Preliminary Results}

The main ingredient in the proofs is the following well known result, which we recall as a courtesy to the reader.

\begin{theorem}[Hermite-Dickson criterion, \cite{MR746963}]

Let $\F_q$ be of characteristic $p$. Then $f \in \F_q[x]$ is a permutation polynomial of $\F_q$ if and only if the following two conditions hold:
\begin{enumerate}
	\item $f$ has exactly one root in $\F_q$,
	\item for each integer $t$ with $1 \leq t \leq q-2$ and $t \not\equiv 0 \pmod p$, the reduction of $f(x)^t \pmod{x^q-x}$ has degree less than or equal to $q-2$.
\end{enumerate}

\end{theorem}

Another known result is the following.

\begin{Lemma}[\cite{MR1505161}, \cite{MR1505164}, \cite{MR1505176}]\label{Lucas}
Let $p$ be a prime and let $\alpha,\beta$ be non-negative integers with base $p$ expansions $\alpha = \sum_i \alpha_i p^i$ and $\beta = \sum_j \beta_j p^j$. Then
$$\binom{\alpha}{\beta} \equiv \prod_i \binom{\alpha_i}{\beta_i} \pmod p,$$
where we use the convention $\binom{n}{k} = 0$, if $n < k$.
\end{Lemma}

We will also need the following lemma, in which we evaluate a certain sum.

\begin{Lemma}\label{lemmakpn}
Let $$S(k,r) = \sum_{i=0}^k \sum_{j=0}^k  (-1)^{i+j}  \binom{k}{i}\binom{k}{j} (i-j)^{r}.$$ Then $S(k,r) = 0$ if $r$ is odd or if $r < 2k$, 
$$S(k,2k) = (-1)^k (2k)!$$ and $$S(k,2k+2) = (-1)^k (2k)! k(k+1)(2k+1)/6 = (-1)^k (2k)! (1 + 2^2 + \cdots + k^2).$$
\end{Lemma}

\begin{proof}

We note that when $r$ is odd the given sum is zero, since the coefficient of $(i-j)^r$ and $(j-i)^r$ is the same.\\
Changing indices of the sum we obtain

$$S(k,r) = \sum_{j=0}^k \sum_{i=-j}^{k-j} (-1)^i \binom{k}{j} \binom{k}{j+i} i^{r} = \sum_j \binom{k}{j} \sum_i (-1)^i \binom{k}{j+i} i^{r}.$$
This is the evaluation at $x = -1$ of the polynomial
$$\sum_{j=0}^k \binom{k}{j} (xD)^{r} \sum_{i=-j}^{k-j} \binom{k}{j+i} x^i = \sum_{j=0}^k \binom{k}{j} (xD)^{r} \frac{(1+x)^k}{x^j},$$
where $(xD)(f) = xD(f)$ and $D$ is the standard derivative. By linearity of this operator we can exchange the sums and we obtain that $S(k,r)$ is the evaluation at $x = -1$ of
$$(xD)^r (1+1/x)^k (1+x)^k = (xD)^r \frac{(1+x)^{2k}}{x^k}.$$
By direct computation we see that $$(xD)\frac{(1+x)^i}{x^j} = \frac{i(1+x)^{i-1}}{x^{j-1}} - \frac{j(1+x)^i}{x^j}.$$
Since the exponent of $(1+x)$ in the numerator decreases at most by one every time we apply the operator $xD$, we have that, if $r < 2k$, the numerator of $(xD)^r \frac{(1+x)^{2k}}{x^k}$ is divisible by $1+x$ and therefore, evaluating at $-1$, we obtain zero.\\
From the above formula we see that applying $k-1$ times the operator we obtain $(2k)_{k-1} \frac{(1+x)^{k+1}}{x} + p$, where $p$ is divisible by $(1+x)^{k+2}$. Now $(xD) (1+x)^i/x = i (1+x)^{i-1} - (1+x)^i/x$. From this we have that $$(xD)^k \frac{(1+x)^{2k}}{x^k} = (2k)_{k} (1+x)^k + q,$$ where $q$ is divisible by $(1+x)^{k+1}$.

Now $(xD) (1+x)^i = i x (1+x)^{i-1}$ and $(xD) (i x (1+x)^{i-1}) = i(i-1)x^2 (1+x)^{i-2} + i x (1+x)^{i-1}$. From this formula we immediately see that 

$$(xD)^{2k} \frac{(1+x)^{2k}}{x^k} = (2k)! x^k + s,$$
where $s$ is divisible by $(1+x)$. Thus, evaluating it at $x=-1$, we obtain $S(k,2k) = (-1)^k (2k)!$.\\
To prove the thesis for $r = 2k+2$ we will use induction. By direct computation we see that the conclusion is true for $k \in \{1,2,3\}$.\\
Taking the second derivative, we have to consider the evaluation at $x=-1$ of

$$(xD)^{2k} \left((xD)^2 \frac{(1+x)^{2k}}{x^k} \right) = (xD)^{2k}  \left(k^2 \frac{(1+x)^{2k}}{x^k} - 2k(2k-1) \frac{(1+x)^{2k-2}}{x^{k-1}} \right).$$
The first term is $k^2S(k,2k) = (-1)^{k}k^2(2k)!$. We apply now the inductive hypothesis on the second term and we obtain $-2k(2k-1)(-1)^{k-1}(2k-2)! (1+2^2 + \cdots + (k-1)^2)$. Summing these two terms we get $(-1)^k (2k)! (1 + 2^2 + \cdots + k^2)$ and the conclusion has been proved.
\end{proof}

The last result is a technical lemma.

\begin{Lemma}\label{nonquad}
Suppose $p \ge 5$ and consider the field extension $\F_{p^2}$ over $\F_p$. Then there exists a non-square element $m$ in $\F_{p^2}$ such that $\mathrm{Norm}(1+m) = 4$.
\end{Lemma}

\begin{proof}
Consider the basis $(1,s)$ of $\F_{p^2}$ over $\F_p$, where $s^2 = t$, a non-square element in the base field. We write $m = m_1 + m_2s$ and let $k = \mathrm{Norm}(m) = m_1^2 - tm_2^2$. We need to find $m_1$ and $m_2$ in $\F_p$ such that $m$ is not a square and $(1 + m_1 + m_2s)(1 + m_1 + m_2s)^p = 4$. If we expand it we obtain the equation $2m_1 + k-3 = 0$. Consider the following system:

\begin{center}
$\begin{cases}
2m_1 + k-3 = 0 \\
m_1^2 - tm_2^2 = k 
\end{cases}$
\end{center} 
 
We need to find a solution ($m_1,m_2,k$) such that $k$ is not a square in $\F_p$. From the first equation we obtain $2m_1 = 3-k$. Substituting in the second one we obtain
$$4tm_2^2 = k^2 - 10k + 9.$$
If we find $k \in \F_p$ such that both $k$ and $k^2 - 10k + 9$ are not squares we are done, since we can take $m_2$ satisfying $m_2^2 = ( k^2 - 10k + 9)/4t$. We note that $f(k) = k^2 - 10k + 9 = (k-1)(k-9)$. Because the equation $k^2-10k+9=h^2$ can be written as
$(k+h-5)(k-h-5)=16$, it has exactly $p-1$ solutions
$(k,h)\in\F_p^2$. Four of them are $(1,0)$, $(9,0)$ and $(0,\pm
3)$. Hence the equation can have at most $p-5$ solutions where $k$
is not a square. Because those come in pairs $(k,\pm h)$, there are
at least two of the $(p-1)/2$ non-squares $k$ in $\F_p$ such that
$k^2-10k+9$ is not a square.
\end{proof}

\section{Proofs}
\label{proofs}
We can now present the proof of the main results, stated in the introduction.

\begin{proof}[Proof of Theorem \ref{teo1}]
We know from the main theorem of \cite{MR1045927} that a polynomial $f$ such that $f(x+a)-f(x)$ is a permutation polynomial for every $a \in \F_q, a \neq 0$, must be a quadratic polynomial. This implies that $\nabla_{a_1,\dots,a_{k-1}}^{k-1} x^n$ has degree $2$. Every time we apply the finite difference operator the degree drops by $1$. Therefore we must have that $n = k+1$. We note that $\nabla_{a_1,\dots,a_{k}}^{k} x^{k+1} = (k+1)!a_1 a_2 \cdots a_k x + c$, for some constant $c$. Thus $x^{k+1}$ is a $k$-PN monomial.
\end{proof}

\begin{proof}[Proof of Theorem \ref{teo2}]
Consider $g = \nabla_{a_1,\dots,a_k}^k f$, where $a_i \in \F_p$. Since $g$ is defined over $\F_p$ we have that $g(\F_p) \subset \F_p$, hence $g$ should be a permutation polynomial of the prime field for every choice of the $k$ directions. From Theorem \ref{teo1} we know that $n \equiv k+1 \pmod{p-1}$. \\
Let now $f(x)$ be the $k$-th derivative (always in direction $1$) of $x^n$. We have $$f(x) = (-1)^k \sum_{i = 0}^{k} (-1)^i \binom{k}{i} (x+i)^n.$$ 
If we write $n = a + bp$ we have two possible cases, $a+b = k+1$ and $a + b = p+k$. We want to exclude the latter. We consider $f^{1+p} \pmod{x^{p^2} - x}$ and show that it has degree $p^2-1$. Then, by the Hermite-Dickson Criterion, $f$ cannot be a permutation polynomial.\\
We have $$ f(x)^{1+p} = \sum_{i=0}^k \sum_{j=0}^k (-1)^{i+j} \binom{k}{i} \binom{k}{j} (x+i)^n (x+j)^{pn}.$$

Consider now the single term $(x+i)^n (x+j)^{pn}$, $i,j \neq 0$. Expanding it we obtain
$$ \sum_{\alpha = 0}^a \sum_{\beta = 0}^b \sum_{\gamma = 0}^a \sum_{\delta = 0}^b \binom{a}{\alpha}\binom{b}{\beta}\binom{a}{\gamma}\binom{b}{\delta} i^{a-\alpha + (b - \beta)p} j^{a - \gamma + (b - \beta)p} x^{\alpha + \delta + (\beta + \gamma)p}.$$
Since $\alpha + \delta + (\beta + \gamma)p < (p + (p-1)/2) + (p  + (p-1)/2)p < 2(p^2-1)$ for $p \geq 5$, we need to consider only the case $\alpha + \delta = \beta + \gamma = p-1$.\\
Hence the coefficient $c$ of the term of degree $p^2-1$ will be

$$\sum_{\alpha = a-k-1}^a \sum_{\beta = b-k-1}^b \binom{a}{\alpha} \binom{b}{\beta} \binom{a}{p-1-\beta}\binom{b}{p-1-\alpha} i^{a - \alpha + b-\beta} j^{a- \gamma + b- \beta}.$$
We rewrite this sum as $C_1 C_2$ where
$$ C_1 = \left( \sum_{\alpha = a-k-1}^a \binom{a}{\alpha} \binom{b}{\alpha + 1+ k-a} i^{a - \alpha} j^{\alpha + 1 + k -a} \right)$$ and $$C_2 = \left( \sum_{\beta = b-k-1}^b \binom{b}{\beta} \binom{a}{\beta + 1+ k-a} i^{b - \beta} j^{\beta + 1 + k - b} \right).$$
Finally we obtain
$$\left( \sum_{l=0}^{k+1} \binom{a}{k+1-l}\binom{b}{l} i^{k+1-l} j^l \right) \left( \sum_{l=0}^{k+1} \binom{b}{k+1-l}\binom{a}{l} i^{k+1-l} j^l \right).$$

Consider the single term $\binom{a}{k+1-l}\binom{b}{l}$. Using the fact that $a + b \equiv k \pmod p$ we have

$$\binom{b}{l} = \frac{b(b-1) \cdots  (b-l+1)}{l!} \equiv (-1)^l \frac{(a-k)(a-k+1) \cdots (a-k+l-1)}{l!}.$$
Therefore $$\binom{a}{k+1-l}\binom{b}{l} \equiv (-1)^l \binom{a}{k+1}\binom{k+1}{l} \pmod p.$$
Finally, we see that $c$ is equivalent modulo $p$ to

$$\left( \binom{a}{k+1} \sum_{l=0}^{k+1} (-1)^l \binom{k+1}{l} i^{k+1-l} j^l \right) \left( \binom{b}{k+1} \sum_{l=0}^{k+1} (-1)^l \binom{k+1}{l} i^{k+1-l} j^l \right).$$
Noting that $\binom{b}{k+1} \equiv (-1)^{k+1} \binom{a}{k+1} \pmod p$ we have that

$$c \equiv (-1)^{k+1} \binom{a}{k+1}^2 (i-j)^{2k+2} \pmod p.$$
Consider now the term $x^n (x+j)^{pn}$. Expanding it we obtain

$$\sum_{\alpha = 0}^a \sum_{\beta = 0}^b \binom{a}{\alpha}\binom{b}{\beta} j^{a - \alpha + (b - \beta)p} x^{a + \beta + (\alpha + b)p}.$$
As before we need only to consider the case $ a + \beta = \alpha + b = p-1$. Thus the coefficient $c$ of the term of degree $p^2-1$ will be

$$ c = \binom{a}{k+1} \binom{b}{k+1} j^{2k+2} \equiv (-1)^{k+1} \binom{a}{k+1}^2 j^{2k+2}.$$
In the same way we see that the monomial of degree $p^2-1$ of $x^{np} (x+j)^{n}$ has the same coefficient.

We have seen that the coefficient modulo $p$ of the term of degree $p^2-1$ of $(x+i)^n (x+j)^{pn} \pmod{x^{p^2} - x}$ equals $(-1)^{k+1} \binom{a}{k+1}^2 (i-j)^{2k+2}$.\\
Summing up all the terms we obtain that the coefficient of the term of degree $p^2-1$ of $f(x)^{1+p} \pmod{x^{p^2} - x}$ is

$$(-1)^{k+1} \binom{a}{k+1}^2 \sum_{i=0}^k \sum_{j=0}^k (-1)^{i+j} \binom{k}{i}\binom{k}{j} (i-j)^{2k+2}.$$
Thanks to Lemma \ref{lemmakpn} this sum is not zero for the given choice of $p \geq 2k+2$.
\end{proof}

\begin{proof}[Proof of Theorem \ref{teo3}]
Let us start with the case $k=2$. If $f = x^3$, its second derivative in directions $d$ and $e$ is $6dex + s$ for some constant $s$ and that is clearly a permutation polynomial. Now, consider the case $n = 2 + p$, the other case being obtained by taking the $p$-th power.

Thus, let $f = x^{2 + p}$. Then we have $\nabla_{d,e} f = 2de x^p + (2de^p + 2d^pe) x + s$ for some constant $s$. We know that a $p$-polynomial is permutation if and only if it has only one zero. Therefore $\nabla_{d,e} f$ is a permutation polynomial if and only if the equation $u^{p-1} + y^{p-1} + z^{p-1} = 0$ has no solution with $u,y,z \in \F_{p^2}^*$. Dividing by $z^{p-1}$ this is equivalent to requiring that the equation $u^{p-1}  + y^{p-1} + 1 = 0$ has no solution with $u,y \in \F_{p^2}^*$. Suppose now $(u,y)$ is such a solution and let $z = u^{p-1}$. Then $y^{p-1} = -1-z$. Since the $(p-1)$-th powers in $\F_{p^2}^*$ are precisely elements of norm $1$ we have $z^{p+1} = (-1-z)^{p+1} = 1$. From these equations we obtain $z^p + z = -1$. Hence $z$ has norm $1$ and trace $-1$. Its minimal polynomial is $z^2 + z + 1$, thus $z$ is a primitive third root of unity and we must have $p \equiv -1 \pmod 3$, since $z$ is not an element of the prime field. Conversely, if $p \equiv -1 \pmod 3$, let $z$ be a primitive third root of unity in $\F_{p^2}$. Since $p-1 \equiv 1 \pmod 3$ we have that $z^{p-1} = z$. Hence, taking $u = z$ and $y = z^p$ we obtain a solution of $u^{p-1} + y^{p-1} + 1=0$.\\

Now, let $k = 3$. If $f = x^4$ we have $\nabla_{c,d,e} f = 24cdex + s$ for some constant $s$ and this is a permutation polynomial. Now, let $n = 3 + p$. We have $\nabla_{c,d,e} x^n = 6cdex^p + (6c^pde + 6cd^pe + 6cde^p)x + s$ for some constant $s$. This is a permutation polynomial if and only if the equation $u^{p-1} + v^{p-1} + w^{p-1} + t^{p-1} = 0$ has no solutions in $\F_{p^2}^*$. But $-1$ has norm $1$, hence there is $z \in \F_{p^2}$ such that $z^{p-1} = -1$. Then $(1,1,z,z)$ is a solution of the previous equation.

The last case is $n = 2 + 2p$. We have $\nabla_{1,1,1} x^n = 12x^p + 12x + s$ for some constant $s$. But now $x^p + x$ is not a permutation polynomial of $\F_{p^2}$ since, as before, there exists $z$ with $z^{p-1} = -1$.\\
For $p = 5$ and $p = 7$ the conclusion of the theorem follows by direct computation.
\end{proof}

\begin{proof}[Proof of Theorem \ref{teo4}]
Let $f = x^n$ and $n = a+bp+cp^2+dp^3$. From Theorem \ref{teo1} we know that $n \equiv 3 \pmod{p-1}$. Hence we have that $a+b+c+d \in \{3, p+2, 2p+1, 3p \}$. We consider each case separately.\\

\textit{Case 1.}
Suppose $a+b+c+d = 3$. We can assume $a \neq 0$, thus we have, up to multiplying by some power of $p$, that $n \in \{3, 2+p, 2+p^2, 2+p^3, 1 + p + p^2 \}$.\\
Suppose $n = 2+p$. We have
$$\nabla_{u,v} f = 2uvx^p + (2uv^p + 2u^pv)x + s,$$
for some constant $s$. This is a permutation polynomial if and only if the equation $y^{p-1} + z^{p-1} + t^{p-1} = 0$ has no solutions in $\F_{p^4}^*$. We recall Weil's bound for the number of $\F_q$-rational projective points $N$ of a smooth curve in $\Proj^2(\F_q)$. We have
$$|N - q - 1| \leq 2g \sqrt{q},$$
where $g = \frac{(d-1)(d-2)}{2}$ is the genus of the curve and $d$ is the degree of the defining polynomial. In our case the lower bound reads $N \geq 5p^3 - 6p^2 + 1$. We are interested in solutions $(y,z,t)$, where $yzt \ne 0$. The equation $y^{p-1} + z^{p-1} = 0$ has $p-1$ (projective) solutions, hence in total we have to exclude $3(p-1)$ solutions from the number obtained before. But $5p^3 - 6p^2 + 1 > 3(p-1)$, therefore we always have solutions in $\F_{p^4}^*$.

Suppose $n = 2+p^3$. We have
$$\nabla_{u,v} f = 2uvx^{p^3} + (2uv^{p^3} + 2u^{p^3}v)x + s,$$
for some constant $s$. This is a permutation polynomial if and only if $y^{p^3-1} + z^{p^3-1} + t^{p^3-1} = 0$ has no solutions in $\F_{p^4}^*$. Since $\gcd(p^3-1,p^4-1) = p-1$ the number of solutions of this equation is the same as in the previous case, hence we can conclude as before.

Suppose $n = 2 + p^2$. We have
$$\nabla_{u,v} f = 2uvx^{p^2} + (2uv^{p^2} + 2u^{p^2}v)x + s,$$
for some constant $s$. This is a permutation polynomial if and only if $y^{p^2-1} + z^{p^2-1} + t^{p^2-1} = 0$ has no solution with $y,z,t \in \F_{p^4}^*$, if and only if $y^{p^2-1} + z^{p^2-1} + 1= 0$ has no solution with $y,z \in \F_{p^4}^*$. Suppose $(y,z)$ is a solution and let $w= y^{p^2-1}$. We have $w^{p^2+1} =1$. We also have $(-1-w)^{p^2+1} = 1$. Substituting the first equation in the last one we obtain $w^{p^2} + w + 1 = 0$. Therefore $w^{p^2} = -1 -z$ and this gives, using the first equation, $w(-1-w) = 1$, i.e. $w^2 + w + 1 = 0$. This is an equation of degree $2$ over $\F_p$, therefore $w \in \F_{p^2}$. But then $w^{p^2} = z$ and $w^{p^2} + w + 1 = 2w + 1$. This is zero if and only if $w = -1/2$. Since $w$ has norm $1$ we must have $2^4 \equiv 1 \pmod p$, therefore $p = 5$. But in $\F_5$ we have that $x^2+x+1$ is irreducible, therefore $w = -1/w \equiv 2 \pmod 5$ cannot satisfy $w^2 + w + 1 = 0$. Hence we do not have a solution and the polynomial is a permutation polynomial.

Suppose now $n = 1 + p + p^2$. We have
$$\nabla_{1,v} f = (v + v^p)x^{p^2} + (v + v^{p^2})x^p + (v^p + v^{p^2})x + s, $$
for some constant $s$. We will now find a solution $(v,x), vx \ne 0$, for the equation $\nabla_{1,v} f -s = 0$ proving that, as a function of $x$, it is not a permutation polynomial. We consider $\F_{p^4}$ as a vector space of dimension $2$ over $\F_{p^2}$ and take $(1,t)$ as basis, with $t^2 = m$, a non-square in $F_{p^2}$.  Let now $v = v_1 + v_2t$ and $x = x_1 + x_2 t$ the decomposition of $v$ and $x$ over our basis. Rewriting the previous equation we obtain the following system of two equations:

\begin{center}
$\begin{cases}
x_1 v_1 + x_1^pv_1+x_1v_1^p - mv_2x_2 = 0 \\
cv_1x_2^p + cx_1v_2^p = 0
\end{cases}$
\end{center}
where $c = t^{p-1} \in \F_{p^2}$.

Set now $v_1 = v_2 = 1$. We have $x_2^p = -x_1$, hence $x_2 = -x_1^p$. Putting it in the first equation we obtain

$$(1+m)x_1^p + 2x_1 = 0.$$
We know that this equation has a solution in $\F_{p^2}$ if and only if $2/(1+m)$ (as element of $\F_{p^2}$) has norm $1$ over $\F_p$. Since by Lemma \ref{nonquad} we can choose a non-square $m$ with this property, we solve the system, hence proving that $x^{1+p+p^2}$ is not $2$-PN over $\F_{p^4}$.\\

\textit{Case 2.}
Suppose now $n = a+bp+cp^2+dp^3$ and $a+b+c+d = p+2$.  We will show that $g   = (\nabla_{1,t}f)^{1+p+p^2+p^3}$ has degree $p^4-1$ for some direction $t$, hence it is not a permutation polynomial. We have $\nabla_{1,t}f = x^n - (x+1)^n - (x+t)^n + (x+1+t)^n$. Hence $g$ will consist of $4^4$ terms of the form
$$(x+i)^n (x+j)^{np} (x+k)^{np^2} (x+l)^{np^3},$$
where $i,j,k,l \in \{0,1,t,1+t \}$. Expanding the previous term using the binomial theorem we obtain
$$\sum_{(\alpha_i),(\beta_i),(\gamma_i), (\delta_i)} s x^{\alpha_1 + \delta_2 + \gamma_3 +  \beta_4 + (\beta_1 + \alpha_2 + \delta_3 + \gamma_4)p + ( \gamma_1 + \beta_2 + \alpha_3 + \delta_4)p^2 + (\delta_1 + \beta_2 + \gamma_3 + \alpha_4)p^3} I,$$
where $I = i^{e_i} j^{e_j} k^{e_k} l^{e_l},0 \leq \alpha_i \leq a, 0 \leq \beta_i \leq b, 0 \leq \gamma_i \leq c, 0 \leq \delta_i \leq d, e_i = a-\alpha_1 + (b-\beta_1)p + (c - \gamma_1)p^2 + (d - \delta_1)p^3, e_j = d - \delta_2 + (a - \alpha_2)p + (b - \beta_2)p^2 + (c-\gamma_2)p^3 , e_k = c - \gamma_3 + (d - \delta_3)p + (a - \alpha_3)p^2 + (b - \beta_3)p^3, e_l = b-\beta_4 + (c - \gamma_4)p +  (d-\delta_4)p^2 + (a - \alpha_4)p^3 $ and
$$s = \prod_{i=1}^{4} \binom{a}{\alpha_i} \prod_{i=1}^{4} \binom{b}{\beta_i} \prod_{i=1}^{4} \binom{c}{\gamma_i} \prod_{i=1}^{4} \binom{d}{\delta_i}.$$
Since $\alpha_1 + \delta_2 + \gamma_3 +  \beta_4 + (\beta_1 + \alpha_2 + \delta_3 + \gamma_4)p + (\gamma_1 + \beta_2 + \alpha_3 + \delta_4)p^2 + (\delta_1 + \beta_2 + \gamma_3 + \alpha_4)p^3 < 2(p^4-1)$, in order to compute the coefficient $M$ of degree $p^4-1$ we need to consider only the terms with $\alpha_1 + \delta_2 + \gamma_3 + \beta_4 = \beta_1 + \alpha_2 + \delta_3 + \gamma_4 = \gamma_1 +  \beta_2 + \alpha_3 + \delta_4 = \delta_1 + \beta_2 + \gamma_3 + \alpha_4 = p-1$.
Thus we have that $M = M_1 M_2 M_3 M_4$, where
$$M_1 = \sum_{\alpha_1 + \delta_2 + \gamma_3 + \beta_4 = p-1} \binom{a}{\alpha_1}\binom{b}{\beta_4}\binom{c}{\gamma_3}\binom{d}{\delta_2} i^{a- \alpha_1}j^{d - \delta_2} k^{c - \gamma_3} l^{b - \beta_4},$$
$$M_2 = \sum_{\alpha_2 + \delta_3 + \gamma_4 + \beta_1 = p-1} \binom{a}{\alpha_2}\binom{b}{\beta_1}\binom{c}{\gamma_4}\binom{d}{\delta_3} i^{p(b- \beta_1)}j^{p(a - \alpha_2)} k^{p(d - \delta_3)} l^{p(c - \gamma_4)},$$
$$M_3 = \sum_{\alpha_3 + \delta_4 + \gamma_1 + \beta_2 = p-1} \binom{a}{\alpha_3}\binom{b}{\beta_2}\binom{c}{\gamma_1}\binom{d}{\delta_4} i^{p^2(c- \gamma_1)}j^{p^2(b - \beta_2)} k^{p^2(a - \alpha_3)} l^{p^2(d - \delta_4)},$$
$$M_4 = \sum_{\alpha_4 + \delta_1 + \gamma_2 + \beta_3 = p-1} \binom{a}{\alpha_4}\binom{b}{\beta_3}\binom{c}{\gamma_2}\binom{d}{\delta_1} i^{p^3(d- \delta_1)}j^{p^3(c - \gamma_2)} k^{p^3(b - \beta_3)} l^{p^3(a - \alpha_4)}.$$

Consider $g = \nabla_{a_1,a_2} f$, where $a_i \in \F_{p^2}$. We have $g(\F_{p^2}) \subset \F_{p^2}$, hence $f$ has to be a $2$-PN monomial over $\F_{p^2}$. Hence we must have, up to multiplying $n$ by some power of $p$ and reducing modulo $x^{p^4} - x$,  $a+c = 2$ and $b+d = p$, $a+c = 1$ and $b+d = p+1$ or $a+c = 0$ and $b+d = p+2$. Except for the last case we can suppose without loss of generality that $a \neq 0$. We then have four possibilities for $n$: $n = bp + (p+2-b)p^3,  3 \leq b \leq p-1$, $n = 1 + bp + (p+1-b)p^3,  2 \leq b \leq p-1$,  $n = 2 + bp + (p-b)p^3,  1 \leq b \leq p-1$ and $n = 1+bp + p^2 + (p-b)p^3, 1 \leq b \leq p-1$. We consider each of them in order. \\

Let $n = bp + (p+2-b)p^3$. We use the previous formula with $a = 0, c=0$ and $d = p+2-b$ and we compute $M$. Let
$S(i,j) = \binom{b}{3}(i-j)^3$. Then we have $M_1 = S(l,j), M_2 = S(i,k)^p, M_3 = S(j,l)^{p^2}$ and $M_4 = S(k,i)^{p^3}$. Therefore we have
$$ M = \binom{b}{3}^4 (i-k)^6 (j-l)^6.$$ We need now to sum all these terms and we obtain that the coefficient of degree $p^4-1$ of $\nabla_{1,1} f$ is $14400\binom{b}{3}^4$. 
For $p \geq 7$ this coefficient is not zero and we are done.\\

Suppose now $n = 1 + bp + (p+1-b)p^3$. Now we use the previous formula with $a = 1, c=0$ and $d = p+1-b$ and we compute $M$. Let
$$S(i,j,k,l) = \binom{b}{3}l^3 - (b-1)\binom{b}{2}l^2j + b\binom{b}{2}lj^2 - \binom{b+1}{3}j^3 + \binom{b}{2} i (l-j)^2.$$
Then we have $M_1 = S(i,j,k,l), M_2 = S(j,k,l,i)^p, M_3 = S(k,l,i,j)^{p^2}$ and $M_4 = S(l,i,j,k)^{p^3}$. Summing all these terms, a computer computation shows that the coefficient of degree $p^4-1$ of $\nabla_{1,t} f$ is $$r_1 = \frac{4}{9}b^4(b-2)(b+1)(b-1)^4(25b^2-25b-59)$$ when $t=1$ and 
$$r_2 = \frac{8}{9}b^4(b-1)^4(1250b^4-2500b^3-4362b^2+5612b+5981)$$ when $t=2$.
We have $r_2(2) = r_2(-1) = -3456$, which is not zero since $p \geq 5$. This implies that we can exclude the cases $b = 2$ and $b = p-1$, because they produce polynomials that are not $2$-PN. Suppose now $3 \leq b \leq p-2$. We will show that $p_1 = 25x^2-25x-59$ and $p_2 = 1250x^4-2500x^3-4362x^2+5612x+5981$ cannot have a common root $b$ in the prime field $\F_p$.  Suppose that $p_1(b) \equiv p_2(b) \equiv 0 \pmod p$. Suppose $p \neq 5$ and consider $p_3 = p_2 - 50x^2p_1+50xp_1 = -2662 x^2+2662x+5981$. We must have $p_3(b) \equiv 0 \pmod p$. Suppose $p \neq 11$. Considering $2662p_1 + 25p_3$ we obtain that $p$ must divide $-7533 = - 3^5 \cdot 31$. A computer computation shows that $p_1$ and $p_2$ are coprime if $p \in \{5,11 \}$. When $p = 31$ their greatest common divisor is $x^2 + 30x +b$, which is irreducible over $\F_{31}$, hence it has no roots in that field. Putting all together we see that, with this choice of $n$, $f$ cannot be $2$-PN.\\

Let now $n = 2 + bp + (p-b)p^3,  1 \leq b \leq p-1$. As before, let
$$S(i,j,k,l) = \binom{b}{3}l^3 - b \binom{b}{2}l^2j+ b \binom{b+1}{2}lj^2 - \binom{b+2}{3}j^3 + $$
$$+ 2i\left(\binom{b}{2}l^2 - b^2 lj + \binom{b+1}{2}j^2 \right) + i^2b(l-j).$$
We have $c_1 = S(i,j,k,l), c_2 = S(j,k,l,i)^p, c_3 = S(k,l,i,j)^{p^2}$ and $c_4 = S(l,i,j,k)^{p^3}$. As in the previous case a computer computation shows that
the coefficient of degree $p^4-1$ of $(\nabla_{1,t} f)^{1+p+p^2+p^3}$ is $$r_1 = \frac{4}{9}b^4(b-1)(b-2)(b+2)(b+1)(25b^4-197b^2+100)$$ when $t = 1$ and
$$r_2 = \frac{16}{9}b^4(625b^8-8518b^6+31641b^4-32452b^2+10648)$$ when $t=2$. We have $r_2(1) = r_2(-1) = 3456$ and $r_2(2) = r_2(-2) = 55296$, which are not zero since $p \geq 5$. With the Euclidean algorithm we see that $25x^4-197x^2+100$ and $625x^8-8518x^6+31641x^4-32452x^2+10648$ are coprime modulo $p$, unless $ p \in \{19, 156797 \}$. If $p = 19$ their greatest common divisor is $x^2+5$, which is irreducible over $\F_{19}$. If $p = 156797$ the greatest common divisor is $x^2
 + 79228$ and this is irreducible over $\F_{156797}$. Hence $x^n$ cannot be $2$-PN with this choice of $n$.\\
 
Finally, let $n = 1+bp + p^3 + (p-b)p^3, 1 \leq b \leq p-1$. Consider
$$S(i,j,k,l) = \binom{b}{3}l^3 - b\binom{b}{2}l^2j + b\binom{b+1}{2}lj^2 - \binom{b+2}{3}j^3 + $$
$$ + (i+k)\left(\binom{b}{2}l^2 - b^2 lj + \binom{b+1}{2} j^2 \right) + ikb(l-j).$$
We have $M_1 = S(i,j,k,l), M_2 = S(j,k,l,i)^p, M_3 = S(k,l,i,j)^{p^2}$ and $M_4 = S(l,i,j,k)^{p^3}$. As in the previous cases a computer computation shows that
the coefficient of degree $p^4-1$ of $(\nabla_{1,t} f)^{1+p+p^2+p^3}$ is $$r_1 = \frac{4}{9}b^4(b-1)^2(b+1)^2(25b^4-2b^2+49)$$ when $t=1$ and
$$r_2 = \frac{16}{9}b^4(625b^8-1138b^6+2238b^4-1834b^2+2053)$$ when $t = 2$. We have $r_2(1) = r_2(-1) = 3456$, which is not zero since $p \geq 5$. With the Euclidean algorithm we see that $625x^8-1138x^6+2238x^4-1834x^2+2053$ and $25x^4 - 2x^2 + 49$ are coprime modulo $p$, unless $p = 12497$. In this case the greatest common divisor is $x^2 + 9356$ and this is irreducible over $\F_{12497}$. Therefore $x^n$ cannot be $2$-PN with this choice of $n$.\\

\textit{Case 3.}
Let $n = a+bp+cp^2+dp^3$ with $a+b+c+d = 2p+1$.  We will show that $g   = (\nabla_{1,1}f)^{1+p^2}$ has degree $p^4-1$, hence it is not a permutation polynomial.
As before we use the binomial theorem to expand $g$, obtaining terms of the form
$$(x+i)^n(x+j)^{np^2} = \sum_{(\alpha_i),(\beta_i),(\gamma_i), (\delta_i)} s x^{\alpha_1 + \gamma_2 + (\beta_1 + \delta_2)p + ( \gamma_1 + \alpha_2)p^2 + (\delta_1 + \beta_2 )p^3} i^{e_i} j^{e_j},$$
where $0 \leq \alpha_i \leq a, 0 \leq \beta_i \leq b, 0 \leq \gamma_i \leq c, 0 \leq \delta_i \leq d, e_i = a-\alpha_1 + (b-\beta_1)p + (c - \gamma_1)p^2 + (d - \delta_1)p^3, e_j = c - \gamma_2 + (d - \delta_2)p + (a - \alpha_2)p^2 + (b-\beta_2)p^3$ and
$$s = \prod_{i=1}^{2} \binom{a}{\alpha_i} \prod_{i=1}^{2} \binom{b}{\beta_i} \prod_{i=1}^{2} \binom{c}{\gamma_i} \prod_{i=1}^{2} \binom{d}{\delta_i}.$$
Since $\alpha_1 + \gamma_2 + (\beta_1 + \delta_2)p + ( \gamma_1 + \alpha_2)p^2 + (\delta_1 + \beta_2 )p^3 < 2(p^4-1)$, in order to compute the coefficient $M$ of degree $p^4-1$ we need to consider only the terms with $\alpha_1 + \gamma_2 = \beta_1 + \delta_2 = \gamma_1 + \alpha_2 = \delta_1 + \beta_2 = p-1$.
Thus we have that $M = M_1 M_2 M_3 M_4$, where
$$M_1 = \sum_{\alpha_1 + \gamma_2 = p-1} \binom{a}{\alpha_1}\binom{c}{\gamma_2} i^{a- \alpha_1}j^{c - \gamma_2},$$
$$M_2 = \sum_{\beta_1 + \delta_2 = p-1} \binom{b}{\beta_1} \binom{d}{\delta_2} i^{p(b- \beta_1)}j^{p(d - \delta_2)},$$
$$M_3 = \sum_{\alpha_2 + \gamma_1 = p-1} \binom{a}{\alpha_2} \binom{c}{\gamma_1} i^{p^2(c- \gamma_1)}j^{p^2(a - \alpha_2)},$$
$$M_4 = \sum_{\delta_1 + \beta_2 = p-1} \binom{b}{\beta_2} \binom{d}{\delta_1} i^{p^3(d- \delta_1)}j^{p^3(b - \beta_2)}.$$

As before we reduce $n \pmod{p^2-1}$ since $f$ should be $2$-PN over the subfield $\F_{p^2}$. Let $a+c = k$ and $b+d = 2p+1-k$. If $3 \leq k \leq p-2$ we have that $n \equiv k+1 + (p+1-k)p \pmod{p^2-1}$ and, with this choice of $n$, $f$ is not $2$-PN. If $p+3 \leq k \leq 2p-2$ we have $n \equiv (k-p) + (2p+2-k)p \pmod{p^2-1}$ and we conclude as before. We have four cases left, i.e. $k \in \{p-1, p, p+1, p+2 \}$. Without loss of generality we consider only $k = p-1$ and $k = p$, the other two being obtained considering $f^{np^2} \pmod{x^{p^4} - x}$.\\
Suppose $a+c = p-1$ and $b+d = p+2$. We see immediately that $M_1 = M_3 = 1$. Then we have that $M_2$ would be equal to
$$\binom{b}{3}i^3 + \binom{b}{2}(p+2-b)i^2j + b \binom{p+2-b}{2}ij^2 + \binom{p+2-b}{3}j^3 \equiv \binom{b}{3}(i-j)^3.$$
Exchanging $i$ and $j$ we obtain $M_4$, which is then equal to $-\binom{b}{3}(i-j)^3 \pmod p$. Hence we obtain $M \equiv -\binom{b}{3}^2 (i-j)^6$. Summing up all the terms in order to obtain the coefficient of degree $p^4-1$ of  $(\nabla_{1,1}f)^{1+p^2}$ we notice that this equals $-\binom{b}{3}^2 S(2,6)$ where $S(k,r)$ is the sum we studied in Lemma \ref{lemmakpn}. Finally we obtain $M \equiv -120\binom{b}{3}^2 \pmod p$. This is nonzero for $p \ge 7$.\\
Now, suppose $a+c = p$ and $b+d = p+1$. Then we have $M_1 = ai + cj \equiv a(i-j) \pmod p$ and $M_3 \equiv -a(i-j) \pmod p$. Expanding $M_2$ we obtain
$$\binom{b}{2}i^2 + b (p+1-b)ij + \binom{p+1-b}{2}j^2 \equiv \binom{b}{2} (i-j)^2 \pmod p.$$
Exchanging $i$ and $j$ we obtain that $M_4 \equiv M_2 \pmod p$. Now $$M = M_1M_2M_3M_4 \equiv -a^2\binom{b}{2}^2 (i-j)^6.$$ As before, considering all the terms, we have that the coefficient of degree $p^4-1$ of  $(\nabla_{1,1}f)^{1+p^2}$ is $-a^2\binom{b}{2}^2 S(2,6) = -120a^2\binom{b}{2}^2$, which is not zero when $p \geq 7$.\\

\textit{Case 4.}
Suppose $a+b+c+d = 3p$. Let $a+c =k$ and $b+d = 3p-k$. We have $p+2 \leq k \leq 2p-2$. Reducing modulo $p^2-1$ we obtain $n \equiv k+1-p + (2p-k+1)p \pmod{p^2-1}$ and, with this choice of $n$, $f$ is not $2$-PN over $\F_{p^2}$.\\

A computer computation shows that the same conclusion of the theorem holds for $p=5$.
\end{proof}

\begin{proof}[Proof of Theorem \ref{teo5}]
Let $f = x^n$ and $n = a+bp+cp^2+dp^3$. We know that $n \equiv 4 \pmod{p-1}$. Hence we have that $a+b+c+d \in \{4, p+3, 2p+2, 3p+1 \}$. We consider each case separately.\\

\textit{Case 1.}
Suppose $a+b+c+d = 4$. Up to multiplying by some power of $p$ and reducing modulo $p^4-1$, we have $n \in \{4,3+p,3+p^2,3+p^3,2+p+p^2,2+p+p^3,2+p^2+p^3,2+2p,2+2p^2,1+p+p^2+p^3 \}$. Reducing modulo $p^2-1$ we can exclude, according to Proposition \ref{teo3}, all these cases except for $n = 4$, $n = 3+p^2$ and $n = 2+2p^2$.\\
Let $n = 3+p^2$. Then we have
$$\nabla_{1,u,v} f = 6uvx^{p^2} + x(6uv + 6uv^{p^2} + 6u^{p^2}v)+ s,$$
for some constant $s$. This is a permutation polynomial if and only if the equation $y^{p^2-1} + z^{p^2-1} + t^{p^2-1} + 1 = 0$ has no solutions with $y,z,t \in \F_{p^4}^*$. We note that elements of the form $y^{p^2-1}$ have norm $1$. In $\F_{p^4}$ we have that $-1$ has norm $1$, hence we can take $w$ such that $w^{p^2-1} = -1$ and $(w,w,1)$ is a solution of the previous equation.\\
Suppose now $n = 2+2p^2$. We have
$$\nabla_{1,1,1} f = 12x^{p^2} + 12x + 36$$
and $x^{p^2} + x$ is not a permutation polynomial of $\F_{p^4}$.\\

\textit{Case 2.}
Let $n = a+bp+cp^2+dp^3$ with  $a+b+c+d = p+3$.  We will show that $g   = (\nabla_{1,1,t}f)^{1+p+p^2+p^3}$ has degree $p^4-1$ for some direction $t$, hence it is not a permutation polynomial. We have $$\nabla_{1,1,t}f = - x^n + (x+t)^n - 2(x+1+t)^n + (x+2+t)^n + 2(x+1)^n - (x+2)^n.$$ Thus $g$ will consist of $6^4$ terms of the form
$$(x+i)^n (x+j)^{np} (x+k)^{np^2} (x+l)^{np^3},$$
where $i,j,k,l \in \{0,t,1+t,2+t,1,2 \}$. We will use the same formulae seen in the proof of Theorem \ref{teo4} in order to compute the coefficient of degree $p^4-1$ of these terms.\\
Consider $g = \nabla_{a_1,a_2,a_3} f$, where $a_i \in \F_{p^2}$. We have $g(\F_{p^2}) \subset \F_{p^2}$, hence $f$ has to be a $3$-PN monomial over $\F_{p^2}$. Hence we must have, up to multiplying $n$ by some power of $p$ and reducing modulo $x^{p^4} - x$,  $a+c = 3$ and $b+d = p$. We suppose without loss of generality that $a \neq 0$. We then have three possibilities for $n$: $n = 3 + bp + (p-b)p^3,  1 \leq b \leq p-1$,  $n = 2 + bp + p^2 + (p-b)p^3,  1 \leq b \leq p-1$ and $n = 1+bp + 2p^2 + (p-b)p^3, 1 \leq b \leq p-1$. We note that the last two are the same modulo multiplying by $p^2$, hence we will consider only one of them.\\
Thus, suppose $n = 3 + bp + (p-b)p^3$. We use the formulae of the previous proof with $a = 3, c=0$ and $d = p-b$ and we compute $M$. Let
$$S(i,j,k,l) = i^3b(l-j)+3i^2 \left(\binom{b}{2}l^2-b^2 lj+\binom{b+1}{2} j^2 \right) + $$ 
$$ + 3i \left(\binom{b}{3} l^3 - b\binom{b}{2}l^2j + b\binom{b+1}{2} lj^2- \binom{b+2}{3}j^3 \right)+ $$ $$ + \binom{b}{4}l^4-b\binom{b}{3}l^3j+\binom{b}{2}\binom{b+1}{2}l^2j^2-b\binom{b+2}{3}lj^3+\binom{b+3}{4}j^4.$$
Then we have $M_1 = S(i,j,k,l), M_2 = S(j,k,l,i)^p, M_3 = S(k,l,i,j)^{p^2}$ and $M_4 = S(l,i,j,k)^{p^3}$. Summing all these terms, a computer computation shows that the coefficient of degree $p^4-1$ of $(\nabla_{1,1,t} f)^{1+p+p^2+p^3}$ is \small $$ r_1 = \frac{1}{4}b^4(1225b^{12}-63280b^{10}+798090b^8-3115120b^6+5525413b^4-5086440b^2+2479248)$$ \normalsize when $t=1$ and \small $$r_2 = 16b^4(1225b^{12}-66745b^{10}+868335b^8-3306955b^6+5775712b^4-5057460b^2+2542752)$$ \normalsize when $t = 2$. With the Euclidean algorithm we see that the two polynomials $h_1$ and $h_2$ of degree $12$ in $r_1$ and $r_2$ are coprime unless $$p \in \{5,7,17,233,239,937,28933,323339 \}.$$  If $p \in \{5,7,239,28933 \}$ the greatest common divisor is irreducible over $\F_p$, therefore it is always nonzero when $b \in \F_p$. Now, if $p = 17$ we have $\gcd(h_1,h_2) = b^2 + 9$ which has roots $\pm 5$. If $p = 233$ then $\gcd(h_1,h_2) = b^2+229 = (b+2)(b-2)$. If $p = 937$ then $\gcd(h_1,h_2) = (b+533)(b+404)$. If $p = 323339$ then $\gcd(h_1,h_2) = (b+9299)(b+314040)$. A direct computer computation, considering different directions, shows that in these cases $\nabla_{1,1,t} f$ is not a permutation polynomial.\\

Suppose now $n = 2 + bp + p^2 + (p-b)p^3$. Let
$$S(i,j,k,l) = i^2kb(l-j)+(i^2 + 2ik) \left(\binom{b}{2}l^2-b^2 lj+\binom{b+1}{2} j^2 \right) + $$ 
$$ + (k+ 2i) \left(\binom{b}{3} l^3 - b\binom{b}{2}l^2j + b\binom{b+1}{2} lj^2- \binom{b+2}{3}j^3 \right)+ $$ $$ + \binom{b}{4}l^4-b\binom{b}{3}l^3j+\binom{b}{2}\binom{b+1}{2}l^2j^2-b\binom{b+2}{3}lj^3+\binom{b+3}{4}j^4.$$
Then we have $M_1 = S(i,j,k,l), M_2 = S(j,k,l,i)^p, M_3 = S(k,l,i,j)^{p^2}$ and $M_4 = S(l,i,j,k)^{p^3}$. Summing all these terms, a computer computation shows that the coefficient of degree $p^4-1$ of $(\nabla_{1,1,t} f)^{1+p+p^2+p^3}$ is 
$$r_1 = \frac{1}{4}b^4(1225b^{12}-9380b^{10}+47270b^8-80972b^6+9881b^4-368744b^2+939856)$$
when $t=1$ and
$$r_2 = 16b^4(1225b^{12}-9170b^{10}+44330b^8-80678b^6+62969b^4-271556b^2+1009744)$$
when $t=2$. With the Euclidean algorithm we see that the two polynomials $h_1$ and $h_2$ of degree $12$ in $r_1$ and $r_2$ are coprime unless
$$p \in \{5,7,19,29,101,41051, 15052321 \}.$$
If $p = 5$ or $p = 7$ the greatest common divisor of $h_1$ and $h_2$ has no roots in $\F_p$, hence for every choice of $b$ it cannot vanish. For $p = 19$ we have $\gcd(h_1,h_2) = (b+3)(b+16)$. For $p = 29$ we have $\gcd(h_1,h_2) = (b+27)(b+2)$. For $p = 101$ we have $\gcd(h_1,h_2) = (b+34)(b+67)$. For $p = 41051$ we have $\gcd(h_1,h_2) = (b+17388)(b+23663)$. For $p = 15052321$ we have $\gcd(h_1,h_2) = (b + 3670586)(b + 11381735)$. As before a direct computation shows that we can exclude also these cases, since $x^n$, with such choice of $n$, is not $3$-PN over $\F_{p^4}$.\\

\textit{Case 3.}
Suppose $n = a+bp+cp^2+dp^3$ with $a+b+c+d = 2p+2$.  We will show that $g   = (\nabla_{1,1,1}f)^{1+p^2}$ has degree $p^4-1$, hence it is not a permutation polynomial. As before we reduce $n$ modulo $p^2-1$ and exclude some cases. Suppose $a+c = k$ and $b+d = 2p+2-k$. If $4 \leq k \leq p-2$ then $n \equiv k+1 + (p+2-k)p \pmod{p^2-1}$
and $f$ is not $3$-PN. If $k = p$ we obtain $n \equiv 1 + 3p \pmod{p^2-1}$ and we exclude this value. If $k = p+1$ we have $n \equiv 2 + 2p \pmod{p^2-1}$ and also this case is not $3$-PN. If $p+4 \leq k \leq 2p-2$ we have $n \equiv k-p + (2p+3-k)p$ and also this case is bad. We have, only two cases left, namely $a+c = p-1$ and $a+c = p+3$. We can suppose without loss of generality that $a+c = p-1$. Using the formulae of Theorem \ref{teo4} we conclude that $M_1 = M_3 = 1$ and
\begin{align*}
M_2 & =  \binom{b}{4}i^4 + \binom{b}{3}(p+3-b)i^3j + \binom{b}{2}\binom{p+3-b}{2}i^2j^2 + \\
&  + b \binom{p+3-b}{3}ij^3 + \binom{p+3-b}{4}j^4 \equiv \binom{b}{4}(i-j)^4,
\end{align*}
where the equivalence is modulo $p$. Exchanging $i$ and $j$ we obtain $M_4$ and thus, as before, the coefficient of degree $p^4-1$ of $(\nabla_{1,1,1}f)^{1+p^2}$ is 
$$\binom{b}{4}^2 S(3,8) = -10080 \binom{b}{4}^2,$$
which is not zero for $p \geq 11$.\\

\textit{Case 4.}
Suppose $a+b+c+d = 3p+1$. Let $a+c =k$ and $b+d = 3p+1-k$. We have $p+3 \leq k \leq 2p-2$. Reducing modulo $p^2-1$ we obtain $n \equiv k+1-p + (2p-k+2)p \pmod{p^2-1}$ and, with this choice of $n$, $f$ is not $3$-PN over $\F_{p^2}$.\\

A computer computation shows that the same conclusion of the theorem holds for $p=5,7$.
\end{proof}

\bibliographystyle{alpha}

\def\cprime{$'$}

\end{document}